\documentclass[12pt]{amsart}
\usepackage{amssymb,amsmath,amsthm}
\usepackage{amscd}
\input amssym.def
\input amssym
\newtheorem{theorem}{Theorem}[section]

\newtheorem{remark}{Remark}[section]
\newtheorem{defi}{Definition}[section]
\newtheorem{prop}{Proposition}[section]

\newcommand{\be}{\begin{equation}}
\newcommand{\ee}{\end{equation}}

\newcommand{\St}[2]{\genfrac{[}{]}{0pt}{}{#1}{#2}}
\newcommand{\st}[2]{\genfrac{\{}{\}}{0pt}{}{#1}{#2}}

\renewcommand{\theequation}{\thesection.\arabic{equation}}
\renewcommand{\thetheorem}{\thesection.\arabic{theorem}}
\numberwithin{equation}{section}
\begin{document}

\title[] {Formal iterated logarithms and exponentials and the Stirling numbers}

\author{Thomas J. Robinson}

%\thanks{}

\begin{abstract}
  We calculate the formal analytic expansions of certain formal
  translations in a space of formal iterated logarithmic and
  exponential variables.  The results show how the algebraic structure
  naturally involves the Stirling numbers of the first and second
  kinds, and certain extensions of these, which appear as expansion
  coefficients.
\end{abstract}

\maketitle

\renewcommand{\theequation}{\thesection.\arabic{equation}}
\renewcommand{\thetheorem}{\thesection.\arabic{theorem}}
\setcounter{equation}{0} \setcounter{theorem}{0}
\setcounter{section}{0}

\section{Introduction} 
In \cite{M} and \cite{HLZ} logarithmic formal calculus was used to set
up certain structures for the treatment of logarithmic intertwining
operators and ultimately logarithmic tensor category theory for modules
for a vertex operator algebra.  A minor footnote in \cite{HLZ}
involved two expansions of certain formal series which yielded a
classical combinatorial identity involving Stirling numbers of the
first kind (see (3.17), \cite{HLZ}), which those authors used to solve
a problem posed in \cite{Lu} (see Remark 3.8 in \cite{HLZ}).  These
series expansions were worked out during the course of a proof of a
logarithmic formal Taylor theorem (see Theorem 3.6, \cite{HLZ}).  A
detailed treatment of an efficient algebraic method to obtain formal
Taylor theorems in great generality was given in \cite{R1}.  This
method was demonstrated on a space involving formal versions of
iterated logarithmic and exponential variables, extending the setting
used in \cite{HLZ}.  The method of proof bypasses any series
expansions.  The purpose of this paper is to calculate an
arithmetically natural class of series expansions in this space of
iterated logarithmic and exponential formal variables.  We shall
calculate two different expansions, each extending the methods used in
the special case calculated in \cite{HLZ}.  Equating the coefficients
of these expansions yields a class of combinatorial identities,
including the special case found in \cite{HLZ} that motivated this
work.  We shall also note the corresponding finite difference
equations that certain of the expansion coefficients satisfy.

The combinatorial identity found in \cite{HLZ} involves the Stirling
numbers of the first kind.  We shall find, from an algebraically
symmetric setting, an analogous identity for the Stirling numbers of
the second kind.  We leave as an open question whether there is a
simple or natural way to predict this fact from the algebra (beyond
heuristics) or any salient features of this result having to do with
relationships between the Stirling numbers of both types and also,
conversely, whether or not the combinatorics might suggest any further
natural generalization or extension of the algebraic setting.

This paper is almost entirely self-contained (in particular, no
knowledge of vertex algebra theory, let alone the theory of
logarithmic intertwining operators, is necessary to read this paper).
We need only one outside result, Theorem \ref{th:itlogrec}, which gives a
recursion formula.  We do develop from scratch the necessary definitions
to fully understand the statement of this recursion, and refer the
reader to \cite{R2} for a full proof.

\section{Stirling numbers and other sundry items recalled}
We recall in this section, and record the notation which we shall be
using for, certain standard formal and combinatorial objects and facts
which we use, or which appear during the course of this work.

For $r \in \mathbb{C}$ and $m \geq 0$ we extend, as usual, the
binomial coefficient notation:
\begin{align*}
\binom{r}{m}=\frac{r(r-1)\cdots(r-m+1)}{m!}.
\end{align*}

We make extensive use of the usual formal exponential and logarithmic
series which we recall here.
\begin{align}
\label{eq:formalexp}
e^{yX}=\sum_{n \geq 0} \frac{(yX)^{n}}{n!},
\end{align}
where $y$ is a formal variable and where $X$ is any formal object such
that each coefficient may be finitely computed.  Similarly
\begin{align}
\label{eq:formallog}
\log (1+X)=\sum_{n \geq 0}\frac{(-1)^{n-1}}{n}X^{n},
\end{align}
where, again, $X$ is any formal object such that each coefficient may
be finitely computed.

\begin{remark} \rm We shall also be using a second type of formal
  exponential and logarithm which we discuss in
  Section~\ref{sec:varch} (See also Remark 3.2 in \cite{HLZ} where the
  presence of two different types of formal logarithm is discussed).
\end{remark}

We shall be recovering certain classical facts about Stirling numbers.
We review these well-known facts here (cf. \cite{LW}, \cite{C}).  Our
results, among other things, show the equivalence of three different
characterizations of the Stirling numbers of the first and second
kinds.  Each of these characterizations naturally arise during the
course of the work in this paper.  If one were working in ignorance of
the classical results it would be natural to use different notation
each time the Stirling numbers reappeared, and only at the end, when
the equivalence is shown to unify the notation.  However, since the
results are classical, we shall anticipate these results and use
standard notation throughout, unless otherwise indicated, for purposes
of readability.
\begin{remark} \rm The Stirling numbers have combinatorial
  interpretations (cf.  \cite{LW}, \cite{C}) but we shall need only
  algebraic expressions for them here.
\end{remark}
The (signless) Stirling numbers of the first kind, denoted by $\St{m}{n}$ 
%${ m \atopwithdelims[] n }$ 
for $0 \leq n \leq m$, may be defined by either of the
following two explicit expressions:
\begin{align}
\St{m}{n}
%{ m \atopwithdelims[] n }
=
\sum_{0 \leq t_{1} < t_{2} < \cdots < t_{m-n}< m}t_{1} \cdots t_{m-n}.
\label{stirl1}
\end{align}
and
\begin{align}
\St{m}{n}
%{ m \atopwithdelims[]n }
=\frac{m!}{n!}
\underset{i_{l} \geq 1}{\sum_{i_{1}+\cdots+ i_{n}=m}}
\frac{1}{i_{1}\cdots i_{n}} \label{eq:stirling1},
\end{align}
where in both (\ref{eq:stirling1}) and (\ref{stirl1}), $\St{0}{0}$ 
%${0 \atopwithdelims[] 0}$ 
is interpreted as $1$ and $\St{m}{0}$ 
%${ m \atopwithdelims[] 0}$ 
is interpreted as $0$ for $m \geq 1$.

We also recall, and it is routine to verify, using
(\ref{eq:stirling1}), that for $n \geq 0$
\begin{align}
\label{stirl1gen}
(\log (1+x))^{n}=\sum_{m \geq n}(-1)^{m-n}\frac{n!}{m!} 
\St{m}{n}
%{m \atopwithdelims[] n} 
x^{m}.
\end{align}

The Stirling numbers of the first kind may also be defined as the
unique solution to the following discrete boundary value problem.
\begin{align}
\label{rec1}
\St{m}{n}
%{ m \atopwithdelims[] n }
=(m-1)
\St{m-1}{n}
%{ m-1  \atopwithdelims[] n }
+
\St{m-1}{n-1}
%{ m-1 \atopwithdelims[] n-1 }
,
\end{align}
for $1 \leq n < m$ with
\begin{align}
\label{rec1b}
%{0 \atopwithdelims{} j }=0 \quad \text{for} \quad j >0,
%\qquad
\St{m}{0}
%{m \atopwithdelims[] 0}
=0 \quad \text{for} \quad m > 0,
\qquad \text{and} \quad
\St{m}{m}
%{m \atopwithdelims[] m}
=1 \quad \text{for} \quad m \geq 0.  
\end{align}

The Stirling numbers of the second kind, denoted by $\st{m}{n}$
for $m,n \geq 0$, may be defined by either of the following two
explicit expressions:
\begin{align}
\st{m}{n}
%{m \atopwithdelims\{\} n}
=
\sum_{0 \leq i_{1} \leq i_{2} \leq \cdots  \leq i_{m-n} 
\leq n}i_{1}i_{2}\cdots i_{m-n}.
\label{stirl2}
\end{align}
and
\begin{align}
\st{m}{n} =
\frac{m!}{n!}
\underset{i_{l} \geq 1}
{
\sum_{i_{1}+i_{2}+\cdots i_{n}=m}
}
\frac{1}{i_{1}!i_{2}!\cdots i_{n}!},
\label{eq:stirling2}
\end{align}
where in both (\ref{eq:stirling2}) and (\ref{stirl2}), $\st{0}{0}$
is interpreted as $1$ and $\st{m}{0}$ is interpreted as $0$ for $m
\geq 1$.

We also recall, and it is routine to verify, using
(\ref{eq:stirling2}), that for $n \geq 0$
\begin{align}
\label{stirl2gen}
(e^{x}-1)^{n}=\sum_{m \geq n}\frac{n!}{m!}\st{m}{n}x^{m}.
\end{align}
The Stirling numbers of the second kind may also be defined as the
unique solution to the following discrete boundary value problem.
\begin{align}
\label{rec2}
\st{m}{n}=n\st{m-1}{n} +\st{m-1}{n-1},
\end{align}
for $1 \leq n < m$ with 
\begin{align}
\label{rec2b}
%{0 \st{}{} n }=0 \quad \text{for} \quad n >0,
%\qquad
\st{m}{0}=0 \quad \text{for} \quad m > 0,
\qquad \text{and} \quad
\st{m}{m}=1 \quad \text{for} \quad m \geq 0.  
\end{align}

\section{The automorphism property}
Our approach begins by considering a certain type of automorphism
property.  Let $A$ be an algebra over $\mathbb{C}$ and let $D$ be a
linear operator on $A$.  We wish to consider iterating the action on
$D$ arbitrarily many times.  Thus it is natural to consider a
generating function.  Let
\begin{align*}
f(x)=\sum_{n \geq 0}f_{n}x^{n}
\end{align*}
where $f_{n} \neq 0 \in \mathbb{C}$ and where $x$ is a formal
variable. We want the coefficents to be nonzero so that we account for
every iteration of $D$ (including the $0$-th iteration) when
considering the operator
\begin{align*}
f(xD):A \rightarrow A [[x]],   
\end{align*}
where $A[[x]]$ are the formal power series over $A$.  Later we shall
also use the notation $A[x]$ to mean the formal polynomials over $A$.
We take as a (logical) jumping off point the question of whether we
can find $f(x)$ to get
\begin{align}
\label{eq:aut}
f(xD)(ab)=(f(xD)a)(f(xD)b), \qquad \text{(the automorphism property)}
\end{align}
for all $a,b \in A$.  
\begin{remark} \rm This question is motivated by many well
  known such operators appearing in formal calculus, for instance, in
  Chapter 2 of \cite{FLM}.  Such an operator will turn out
  to be an analogue of an element of a one-parameter Lie group, but we
  shall not need any knowledge of such groups in this work.
\end{remark} 
The identity (\ref{eq:aut}) consists of a sequence of identities, one
for each power of $x$.  The first two identities go as thus
\begin{align*}
f_{0}ab&=f_{0}^{2}ab\\
f_{1}D(ab)&=f_{1}f_{0}(Da)b+f_{0}f_{1}a(Db).
\end{align*}
The first identity forces $f_{0}=1$ since we insisted it be nonzero.
Because $f_{1} \neq 0$ the second identity now gives us
\begin{align*}
D(ab)=(Da)b+a(Db),
\end{align*}
or in other words, it shows that $D$ must be a {\it derivation}. We
may use this fact to expand $D^{m}(ab)$ into terms of the form
$D^{n}aD^{m}b$ and because of this it is easily seen to be convenient
(upon inspection after writing out the next few identities using the
derivation expansion) to make the following assumption: given $M,N
\geq 0$ there exists $a,b \in \mathbb{C}$ such that the elements
$D^{n}aD^{m}b$ for $0 \leq n \leq N$ and $0 \leq m \leq M$ are
linearly independent.  Then we are forced to equate coefficients in an
easy fashion.  In the $n$-th identity, we may, for instance, equate
the coefficients of $(D^{n-1}a)b$.  It is easy to see that this gives
\begin{align*}
nf_{n}=f_{n-1}f_{1},
\end{align*}
for $n \geq 1$.  It is further easy to see that therefore
\begin{align*}
f_{n}=\frac{f_{1}^{n}}{n!},
\end{align*}
for $n \geq 1$.  Thus, in looking for an operator satisfying the
automorphism property we are now reduced to the following possibility:
\begin{align*}
f(xD)=\sum_{n \geq 0}\frac{(f_{1}xD)^{n}}{n!},
\end{align*}
where $D$ is a derivation.  
%
%This motivates from scratch the following
%well-known notation:
%\begin{align*}
%e^{y\square}=\sum_{n \geq 0}\frac{(y\square)^{n}}{n!},
%\end{align*}
%where $\square$ is any formal object for which the expansion makes
%sense (i.e. each term in $y$ may be finitely calculated).  In this
%notation the operators we are interested in may be written as
Recalling (\ref{eq:formalexp}), we have
\begin{align*}
f(xD)=e^{f_{1}xD}.
\end{align*}
In fact, these operators do satisfy the automorphism property as we
verify next.
\begin{prop} (The ``automorphism property'') Let $A$ be an algebra
over $\mathbb{C}$ and let $\alpha \in \mathbb{C}$.  Let $D$ be a
derivation on $A$.  Then
\begin{align*}
e^{\alpha yD}(ab)=\left(e^{\alpha yD}a\right)\left(e^{\alpha yD}b\right).
\end{align*}
\end{prop}

\begin{proof}
  If $\alpha=0$ the statement is trivial.  Expanding each side and
  equating coefficients leads to the identities
\begin{align*}
D^{k}(ab)=\underset{n,m \geq 0}{\sum_{n+m=k}}\frac{k!}{n!m!}D^{n}aD^{m}b,
\end{align*}
for $k \geq 0$.  The $k=0$ identity is trivial and the $k=1$ identity
is the derivation property of $D$.  We proceed by induction assuming
that we have verified the statement up to $k$.  Then
\begin{align*}
D^{k+1}(ab)&=D(D^{k}(ab))\\
&=D\underset{n,m \geq 0}{\sum_{n+m=k}}\frac{k!}{n!m!}D^{n}aD^{m}b\\
&=\underset{n,m \geq 0}{\sum_{n+m=k}}\frac{k!}{n!m!}D^{n+1}aD^{m}b
\quad +\underset{n,m \geq 0}{\sum_{n+m=k}}\frac{k!}{n!m!}D^{n}aD^{m+1}b\\
&=\underset{n,m \geq 1}{\sum_{n+m=k+1}}\frac{k!}{(n-1)!m!}D^{n}aD^{m}b\\
&\quad +\underset{n,m \geq 1}{\sum_{n+m=k+1}}\frac{k!}{n!(m-1)!}D^{n}aD^{m}b
+D^{k+1}a+D^{k+1}b\\
&=\underset{n,m \geq 1}{\sum_{n+m=k+1}}
\left(\frac{k!}{(n-1)!m!}+\frac{k!}{n!(m-1)!}\right)D^{n}aD^{m}b
+D^{k+1}a+D^{k+1}b\\
&=\underset{n,m \geq 0}{\sum_{n+m=k+1}}
\frac{(k+1)!}{n!m!}D^{n}aD^{m}b.\\
\end{align*}
\end{proof}
We shall content ourselves with the case $\alpha=1$ for the remainder
of the paper.
\begin{remark} \rm The preceding proposition is well known (cf.
  \cite{LL}).  It may be proved more easily if one assumes more
  information, for instance, about the binomial expansion (see e.g.
  the proof of Proposition 2.1 in \cite{R1}).  However, assuming this
  knowledge at this stage seems out of order philosophically as we
  shall essentially (re)-prove the binomial theorem as a special
  ``base'' case later on.  Moreover, it is a philosophical point that
  it is ``good'' (vaguely ``more intrinsic'') to take the automorphism
  property as primitive and later obtain combinatorial identities as
  ``automatic'' results, in this case arising from different
  calculations of certain formal expansion coefficients.
\end{remark}
\section{Formal iterated logarithms and exponentials}
\label{sec:varch}
Let $\ell_{n}(x)$ be formal commuting variables for $n \in
\mathbb{Z}$. We consider the algebra with an underlying vector space
basis consisting of all elements of the form
\begin{align*}
\prod_{i \in \mathbb{Z}}\ell_{i}(x)^{r_{i}},
\end{align*}
where $r_{i} \in \mathbb{C}$ for all $i \in \mathbb{Z}$, and all but
finitely many of the exponents $r_{i}=0$.  The multiplication is the
obvious one (when multiplying two monomials simply add the
corresponding exponents and linearly extend).  We call this algebra
\begin{align*}
\mathbb{C}\{[\ell]\}.
%\mathbb{C}
%\{[\dots,\ell_{-1}(x),
%\ell_{0}(x),\ell_{1}(x),\dots]\}.
\end{align*}
We let $\frac{d}{dx}$ be the unique derivation on $\mathbb{C}\{[\ell]\}$
%(which for short we denote by $\mathbb{C}\{[\ell^{\pm 1}]\}$) 
satisfying
\begin{align*}
\frac{d}{dx}\ell_{-n}(x)^{r}
&=r\ell_{-n}(x)^{r-1}\prod_{i=-1}^{-n}\ell_{i}(x),\\
\frac{d}{dx}\ell_{n}(x)^{r}
&=r\ell_{n}(x)^{r-1}\prod_{i=0}^{n-1}\ell_{i}(x)^{-1},\\
\text{\rm and   } \qquad \frac{d}{dx}\ell_{0}(x)^{r}
&=r\ell_{0}(x)^{r-1},
\end{align*}
for $n > 0$ and $r \in \mathbb{C}$.  
\begin{remark} \rm
\label{rem:logexpsecret}
Secretly, $\ell_{n}(x)$ is the $(-n)$-th iterated
exponential for $n < 0$ and the n-th iterated logarithm for $n > 0$
and $\ell_{0}(x)$ is $x$ itself.
\end{remark}
\begin{remark} \rm
To see that this does indeed uniquely define a derivation, we note that
$\frac{d}{dx}$ must coincide with the unique linear map satisfying
\begin{align*}
\frac{d}{dx}\prod_{i \in \mathbb{Z}}\ell_{i}(x)^{r_{i}}
=\sum_{j \in \mathbb{Z}}\frac{d}{dx}\ell_{j}(x)^{r_{j}}
\prod_{i \neq j \in \mathbb{Z}}\ell_{i}(x)^{r_{i}},
\end{align*}
on a basis of $\mathbb{C}\{[\ell]\}$.  This establishes uniqueness.
We need to check that this linear map is indeed a derivation.  It is
routine and we leave it to the reader to verify that it is enough to
check that
\begin{align*}
\frac{d}{dx}(ab)=\left(\frac{d}{dx}a\right)b+a\left(\frac{d}{dx}b\right),
\end{align*}
for basis elements $a$ and $b$.  Another routine calculation reduces
the case to where $a=\ell_{i}(x)^{r}$ and $b=\ell_{i}(x)^{s}$ for $r,s
\in \mathbb{C}$.  Checking this case is trivial once one notes that
\begin{align*}
\frac{d}{dx}\ell_{j}(x)^{r}=r\ell_{j}(x)^{r-1}\frac{d}{dx}\ell_{j}(x).
\end{align*}
\end{remark}
If we let $x$ and $y$ be independent formal variables, then the formal
exponentiated derivation $e^{y\frac{d}{dx}}$, defined by the
expansion, $\sum_{k \geq 0}y^{k}\left(\frac{d}{dx}\right)^{k}/k!$,
acts on a (complex) polynomial $p(x)$ as a formal translation in $y$.
That is, as the reader may easily verify, we have
\begin{align}
e^{y\frac{d}{dx}}p(x)=p(x+y).\label{eq:PFTT}
\end{align}
This motivates the following definition (as in \cite{R1}).
\begin{defi}
Let
\begin{align*}
\ell_{n}(x+y)=e^{y\frac{d}{dx}}\ell_{n}(x) \qquad \text{for} 
\quad n \in \mathbb{Z}.
\end{align*}
\end{defi}

\begin{prop}
(The iterated exponential/logarithmic formal Taylor theorem)\\ For $p(x)
\in \mathbb{C}\{[\ell]\}$ we have:
\begin{align*}
e^{y\frac{d}{dx}}p(x)=p(x+y).
\end{align*} 
\end{prop}

\begin{proof}
The result follows from the automorphism property.
\end{proof}

\section{Formal analytic expansions: method 1}
In this section we begin to calculate formal analytic expansions of
$\ell_{N}(x+y)^{r}$ for $N \in \mathbb{Z}$, $r \in \mathbb{C}$.
Recall Remark \ref{rem:logexpsecret} for the ``true meanings'' of
these objects.  We shall begin with the cases when $N \geq 0$.  We
shall proceed step by step, first handling the cases $N=0$ and $N=1$
separately.  The reader may skip ahead to Subsection
\ref{subsection:itlog1} without any loss of generality.
\subsection{Case $N=0$, the case ``$x$''}
It is easy to see how the $m$-th power of $\frac{d}{dx}$ acts on
$\ell_{0}(x)^{r}$ because
$\frac{d}{dx}\ell_{0}(x)^{r}=r\ell_{0}(x)^{r-1}$, which is a monomial.
This is the essential observation that this method is based on, namely
focusing on such ``isolated'' monomials.  Later we shall have to
expand $\frac{d}{dx}$ as a sum of linear operators yielding such
monomial results, but in this case it is already immediate that for $r
\in \mathbb{C}$ and $m \geq 0$
\begin{align*}
\left(\frac{d}{dx}\right)^{m}\ell_{0}(x)^{r}
=r(r-1)\cdots(r-m+1)\ell_{0}(x)^{r-m}.
\end{align*}
%It will be convenient, as usual, to call
%\begin{align*}
%\binom{r}{m}=\frac{r(r-1)\cdots(r-m+1)}{m!},
%\end{align*}
%for all $r \in \mathbb{C}$.  Then we have the following:
\begin{prop}
For all $r \in \mathbb{C}$
\begin{align*}
e^{y\frac{d}{dx}}\ell_{0}(x)^{r}=\sum_{m \geq
0}\binom{r}{m}\ell_{0}(x)^{r-m}y^{m}.
\end{align*}
\end{prop}
\begin{flushright} $\qed$ \end{flushright}
\subsection{Case $N=1$, the case ``$\log x$''}
We closely follow the argument leading to (3.15) in \cite{HLZ}.  We
have for $r,s \in \mathbb{C}$
\begin{align*}
\frac{d}{dx}\ell_{0}(x)^{s}\ell_{1}(x)^{r}
=s\ell_{0}(x)^{s-1}\ell_{1}(x)^{r}+r\ell_{0}(x)^{s-1}\ell_{1}(x)^{r-1}.
\end{align*}
 Then define two linear operators $T_{0}$ and $T_{1}$ on
 $\mathbb{C}[\ell_{0}(x),\ell_{1}(x)]$ by
\begin{align*}
T_{0}\ell_{0}(x)^{s}\ell_{1}(x)^{r}&=s\ell_{0}(x)^{s-1}\ell_{1}(x)^{r}\\
T_{1}\ell_{0}(x)^{s}\ell_{1}(x)^{r}&=r\ell_{0}(x)^{s-1}\ell_{1}(x)^{r-1}.
\end{align*}
Then 
\begin{align*}
\left(\frac{d}{dx}\right)^{m}\ell_{0}(x)^{s}\ell_{1}(x)^{r}
=(T_{0}+T_{1})^{m}\ell_{0}(x)^{s}\ell_{1}(x)^{r}.
\end{align*}
It is not hard to see that
\begin{align*}
\left(\frac{d}{dx}\right)^{m}\ell_{0}(x)^{s}\ell_{1}(x)^{r}
&=\sum_{j=0}^{m}r(r-1)\cdots(r-j+1) \cdot\\
&\quad \cdot \left(\sum_{0 \leq t_{1} <t_{2} 
<\cdots<t_{m-j}<m}(s-t_{1})\cdots(s-t_{m-j})
\right)
\ell_{0}(x)^{s-m}\ell_{1}(x)^{r-j},
\end{align*}
where the reader should think of $j$ as corresponding to the number of
$T_{1}$'s in a summand of the expansion of $(T_{0}+T_{1})^{m}$ and the
$t_{i}$'s as corresponding to the positions of the $T_{0}$'s.  Thus we
have the following
\begin{prop}
For all $r \in \mathbb{C}$
\begin{align}
\label{eq:method3 log}
e^{y\frac{d}{dx}}\ell_{1}(x)^{r}
&=\sum_{m \geq 0}\left(\frac{y}{\ell_{0}(x)}\right)^{m}
\sum_{j=0}^{m}\binom{r}{j}\ell_{1}(x)^{r-j}
\frac{j!}{m!} (-1)^{m-j}\cdot \nonumber\\
& \quad \cdot \left(\sum_{0 \leq t_{1} <t_{2} <\cdots<t_{m-j}<m}
t_{1}\cdots t_{m-j}\right).
\end{align}
\end{prop}
\begin{flushright} $\qed$ \end{flushright}
\subsection{Case $N \geq 0$, ``iterated logarithms''}
\label{subsection:itlog1}
We begin by defining the following
operators:
\begin{align*}
T_{i}=\prod_{j=0}^{i-1}\ell_{j}(x)^{-1}
\frac{\partial}{\partial \ell_{i}(x)} \qquad i \geq 0.
\end{align*}
Then 
\begin{align*}
\frac{d}{dx}=\sum_{i \geq 0}T_{i}.
\end{align*}
Therefore, 
\begin{align}
\label{def:Tlog}
\left(\frac{d}{dx}\right)^{m}=\sum_{i_{1},i_{2},\cdots, i_{m} \geq 0}
T_{i_{1}}T_{i_{2}}\cdots T_{i_{m}}.
\end{align} 
If we consider all the monomials with a fixed number of occurrences of
each $T_{i}$, and call this fixed number $j_{i}$, then we can
partially calculate to get for $c_{l} \in \mathbb{C}$ 
\begin{align*}
\left(\frac{d}{dx}\right)^{m}\prod_{l=0}^{N}(\ell_{l}(x))^{c_{l}}
=\underset{0 \leq j_{0},j_{1},\cdots, j_{N}}
{\sum_{j_{0}+j_{1}+\cdots+j_{N}=m}}
P(c)
\prod_{i=0}^{N}\ell_{i}(x)^{c_{i}-\alpha_{i}},
\end{align*}
where $P(c)$ is a certain sum of polynomials in the $c_{i}$ each of
degree $j_{i}$ in $c_{i}$, and where, $\alpha_{i}$ is given by
\begin{align}
\label{alpha}
\alpha_{i}=j_{i}+\cdots+j_{N}.
\end{align}

We shall describe $P(c)$ by using a combinatorial construction, a type
of tableau.  A tableau will consist of a specified number of columns
of blank entries each of a specified length.  We shall construct a
tableau on any such ``grid'' of blanks by filling in each blank with
nonnegative numbers beginning at the top of each column and moving
down.  Each new entry can be any nonnegative number subject to two
restrictions.  First, the numbers must strictly ascend as one descends
a column and second, the each entry must be less than or equal to the
number of entries above and to the right (not necessarily above).  So
for example,
$$\shortstack[c]{0\\2 \\8 \\11} \quad 
\shortstack[c]{5 \\6 \\7} \quad
\shortstack[c]{0\\1 \\2 \\3\\ 4}$$
is a tableaux.  But  
$$\shortstack[c]{0\\8 \\2 \\11} \quad 
\shortstack[c]{3 \\5 \\6} \quad
\shortstack[c]{0\\1 \\2 \\3\\ 4}\qquad \text{and} \qquad
\shortstack[c]{0\\2 \\8 \\11} \quad 
\shortstack[c]{3 \\5 \\20} \quad
\shortstack[c]{0\\1 \\2 \\3\\ 4}
$$
are not.

We shall consider all tableaux of a particular shape and assign to
that shape a polynomial in as many variables, $x_{i}$, as there are
columns.  We shall denote this polynomial by 
$$[m_{1},m_{2},\cdots,m_{n}]_{1}(x_{i}),$$ 
where the shape is $n$ columns of heights $m_{1}$
on the left followed by $m_{2}$ next to the right etc., where $m_{i}
\geq 0$.  The polynomial is found by summing over all the tableau of
the given shape.  Each summand is found by inserting ``$x_{i}-$'' in
each entry of the $i$-th column (reading left to right) and
multiplying all entries.  
\begin{prop}
\label{prop:itlogintermediate}
For all $m \geq 0$, $c_{i} \in \mathbb{C}$
\begin{align}
\left(\frac{d}{dx}\right)^{m}\prod_{l=0}^{N}(\ell_{l}(x))^{c_{l}}
=\underset{0 \leq j_{0},j_{1},\cdots, j_{N}}
{\sum_{j_{0}+j_{1}+\cdots+j_{N}=m}}
[j_{0},j_{1},\cdots, j_{N}]_{1}(c_{i})
\prod_{i=0}^{N}\ell_{i}(x)^{c_{i}-\alpha_{i}} \label{eq:D}.
\end{align}
\end{prop}
\begin{proof}
It is easy to see that each summand of (\ref{def:Tlog}) when acting on
$\prod_{l=0}^{N}(\ell_{l}(x))^{c_{l}}$ will yield a coefficient
corresponding to one of the tableaux.  If one 'constructs' each term
of (\ref{def:Tlog}) for a fixed $j_{0},\dots, j_{N}$ by first writing
down all the $T_{N}$'s and then inserting the $T_{N-1}$'s from right
to left etc. in all possible ways then it is easy to see the relevant
one-to-one correspondence between the coefficients yielded by each
term of (\ref{def:Tlog}) and the tableaux of a fixed shape given by
the $j_{i}$'s.
\end{proof}
Notice the way that writing down a term of (\ref{def:Tlog})
corresponded to filling in a tableau of a fixed shape.  The labelling
of each column depended only on the total number of entries there were
in the columns to its right.  That is, when constructing a tableau of
fixed shape each column is labelled independently of the others.  For
instance, the rightmost column of a tableau is completely determined
by its length alone.  Therefore the piece of the `tableau polynomial'
due to the rightmost column may be factored out from all terms
corresponding to a fixed tableau shape.  If we look at the remaining
factor of $[m_{1},m_{2}, \cdots ,m_{n}]_{1}(x_{i})$ and set all the
variables to zero, we get an integer which we shall call
$$(m_{1},m_{2}, \cdots ,m_{n-1};m_{n})_{1}.$$  
The independence of labeling gives immediately the following.
\begin{prop}
\label{prop:inditstirl1}
For $m_{1}, \dots, m_{n} \geq 0$,
$$(m_{1},m_{2}, \cdots, m_{n-1};m_{n})_{1} =(m_{1};m_{2} +\cdots +
m_{n})_{1}\cdots(m_{n-2};m_{n-1}+m_{n})_{1}(m_{n-1};m_{n})_{1}.$$
\end{prop}
\begin{flushright} $\square$ \end{flushright}
It is easy to calculate from the definition that
\begin{align*}
  (m-j:j)_{1}=(-1)^{m-j}\sum_{0 \leq t_{1} < t_{2} < \cdots < t_{m-j}<
    m}t_{1} \cdots t_{m-j}.
\end{align*}
By (\ref{stirl1}) we have
\begin{align*}
(m-j:j)_{1}=(-1)^{m-j}
\St{m}{j}
%{m \atopwithdelims[] j}
.
\end{align*} 
It is easy to see that
%\begin{remark} \rm
%The reader may now recognize a formula for the Stirling numbers of the
%first kind.  Indeed we have
%$$(m;n)=(-1)^{m}
%\begin{bmatrix}
%m+n \\
%n \\
%\end{bmatrix},$$
%a fact which will actually be reproven below, but for ease of notation
%we will anticipate the result.
%\end{remark}
we can specialize \eqref{eq:D} to get
\begin{align*}
\left(\frac{d}{dx}\right)^{m}\ell_{N}(x)^{r}
&=\underset{0 \leq j_{0},\dots, j_{N}}
{\sum_{j_{0}+\cdots+j_{N}=m}}
j_{N}!\binom{r}{j_{N}}
(j_{0}, \cdots ,j_{N-1};j_{N})_{1}
\ell_{N}(x)^{r}\prod_{i=0}^{N}\ell_{i}(x)^{-\alpha_{i}}\\
&=\underset{0 \leq j_{0},\dots, j_{N}}
{\sum_{j_{0}+\cdots+j_{N}=m}}
j_{N}!\binom{r}{j_{N}}
\left(\prod_{i=0}^{N-1}
(j_{i};\alpha_{i+1})_{1}\right)
\ell_{N}(x)^{r}\prod_{i=0}^{N}\ell_{i}(x)^{-\alpha_{i}}.\\
&=\underset{0 \leq j_{0},\dots, j_{N}}
{\sum_{j_{0}+\cdots+j_{N}=m}}
j_{N}!\binom{r}{j_{N}}
(-1)^{\alpha_{0}-\alpha_{N}}
\left(\prod_{i=0}^{N-1}
\St{\alpha_{i}}{\alpha_{i+1}}
%{\alpha_{i} \atopwithdelims[] \alpha_{i+1}}
\right)
\ell_{N}(x)^{r}\prod_{i=0}^{N}\ell_{i}(x)^{-\alpha_{i}}.
\end{align*}
Thus we get 
\begin{theorem} For $r \in \mathbb{C}$, $N \geq 0$,
\label{itlogexpfrst}
\begin{align*}
\ell_{N}(x+y)^{r}
%=e^{y\frac{d}{dx}}\ell_{N}(x)^{r}
=\sum_{m \geq 0}\frac{y^{m}}{m!}
\underset{0 \leq j_{0},j_{1},\cdots, j_{N}}
{\sum_{j_{0}+j_{1}+\cdots+j_{N}=m}}
j_{N}!\binom{r}{j_{N}}
(-1)^{\alpha_{0}-\alpha_{N}}
\left(\prod_{i=0}^{N-1}
\St{\alpha_{i}}{\alpha_{i+1}}
%{\alpha_{i} \atopwithdelims[] \alpha_{i+1}}
\right)
\ell_{N}(x)^{r}\prod_{i=0}^{N}\ell_{i}(x)^{-\alpha_{i}}.
\end{align*}
\end{theorem}
\begin{flushright} $\square$ \end{flushright}

\subsection{Case $N=-1$, the case ``${\rm exp}\;x$''}
We shall next handle the cases $N \leq 0$ step by step, first handling
the cases $N=-1$ and $N=-2$ separately.  The reader may skip ahead to
Subsection \ref{sectioncaseexp} without any loss of generality.

  For $n,m \geq 0$ it is easy to
see how the $m$-th power of $\frac{d}{dx}$ acts on $\ell_{-1}(x)^{r}$
because $\frac{d}{dx}\ell_{-1}(x)^{r}=r\ell_{-1}(x)^{r}$ is a
monomial.  So we have
$\frac{d}{dx}^{m}\ell_{-1}(x)^{r}=r^{m}\ell_{-1}(x)$ which gives the
following:
\begin{prop}
For all $r \in \mathbb{C}$
\begin{align*}
e^{y\frac{d}{dx}}\ell_{-1}(x)^{r}
&=\sum_{m \geq 0}\frac{(ry)^{m}}{m!}\ell_{-1}(x)^{r}\\
&=\ell_{-1}(x)^{r}e^{ry}.
\end{align*}
\end{prop}
\begin{flushright} $\qed$ \end{flushright}

\subsection{Case $N=-2$, the case ``${\rm exp}\;{\rm exp}\;x$''}

We have for $r,s \in \mathbb{C}$
\begin{align*}
\frac{d}{dx}\ell_{-1}(x)^{s}\ell_{-2}(x)^{r}
=s\ell_{-1}(x)^{s}\ell_{-2}(x)^{r}+r\ell_{-1}(x)^{s+1}\ell_{-2}(x)^{r}.
\end{align*}
 Then define two linear operators $S_{0}$ and $S_{1}$ on
 $\mathbb{C}[\ell_{-1}(x),\ell_{-2}(x)]$ by
\begin{align*}
S_{0}\ell_{-1}(x)^{s}\ell_{-2}(x)^{r}
=s\ell_{-1}(x)^{s}\ell_{-2}(x)^{r}\\
S_{1}\ell_{-1}(x)^{s}\ell_{-2}(x)^{r}
=r\ell_{-1}(x)^{s+1}\ell_{-2}(x)^{r}.
\end{align*}
Then 
\begin{align*}
\left(\frac{d}{dx}\right)^{m}\ell_{-1}(x)^{s}\ell_{-2}(x)^{r}
=(S_{0}+S_{1})^{m}\ell_{-1}(x)^{s}\ell_{-2}(x)^{r}.
\end{align*}
It is not hard to see that
\begin{align*}
\left(\frac{d}{dx}\right)^{m}\ell_{-1}(x)^{s}\ell_{-2}(x)^{r}
&=\sum_{j=0}^{m}r^{j}
\left(\sum_{0 \leq t_{1} \leq t_{2} \leq \cdots \leq t_{m-j}\leq j}
(s+t_{1})\cdots(s+t_{m-j})
\right) \cdot \\
& \quad \cdot \ell_{-1}(x)^{s+j}\ell_{-2}(x)^{r},
\end{align*}
where the reader should think of $j$ as corresponding to the number of
$S_{1}$'s in a summand of the expansion of $(S_{0}+S_{1})^{m}$ and the
$t_{i}$'s as corresponding to the number of $S_{1}$'s to the right of
the $i$-th $S_{0}$.  It is now easy to get the following:
\begin{prop}
For all $r \in \mathbb{C}$
\begin{align}
\label{eq:method3 expexp}
e^{y\frac{d}{dx}}\ell_{-2}(x)^{r}
=\ell_{-2}(x)^{r}\sum_{m \geq 0}\frac{y^{m}}{m!}
\sum_{j=0}^{m}(r\ell_{-1}(x))^{j}
\sum_{0 \leq t_{1} \leq t_{2} \leq \cdots \leq t_{m-j}\leq j}
t_{1}\cdots t_{m-j}.
\end{align}
\end{prop}
\begin{flushright} $\qed$ \end{flushright}

\subsection{Formal analytic expansions: 
Cases $N < 0$, ``iterated exponentials''}
\label{sectioncaseexp}
We begin by defining the following
operators:
\begin{align*}
T_{i}=\left(\prod_{j=1}^{i}\ell_{-j}(x)\right)
\frac{\partial}{\partial \ell_{-i}(x)} \qquad i > 0.
\end{align*}
Then for our purposes
\begin{align*}
\frac{d}{dx}=\sum_{i > 0}T_{i}.
\end{align*}
Therefore, 
\begin{align*}
\left(\frac{d}{dx}\right)^{k}=\sum_{i_{1},i_{2},\cdots, i_{k} > 0}
T_{i_{1}}T_{i_{2}}\cdots T_{i_{k}}.\end{align*} 
If we consider all the monomials with a fixed number of occurrences of
each $T_{i}$, and call this fixed number $j_{i}$, then we can
partially calculate to get
\begin{align*}
\left(\frac{d}{dx}\right)^{k}\prod_{l=1}^{N}(\ell_{-l}(x))^{c_{l}}
=\underset{0 \leq j_{1},\cdots, j_{N}}
{\sum_{j_{1}+\cdots+j_{N}=k}}
P(c)
\ell_{-N}(x)^{c_{N}}
\prod_{i=1}^{N-1}\ell_{-i}(x)^{c_{i}+\alpha_{i+1}},
\end{align*}
where $P(c)$ is a certain sum of polynomials in the $c_{i}$ each of
degree $j_{i}$ in $c_{i}$, and where, $\alpha_{i}$ is given by
(\ref{alpha}).
%for notational convenience, we
%have let
%\begin{align*}
%\alpha_{i}=\sum_{l=i+1}^{N}j_{l}.
%\end{align*}

We shall describe $P(c)$ by using a combinatorial construction, a type
of tableau.  A tableau will consist of a specified number of columns
of blank entries each of a specified length.  We shall construct a
tableau on any such ``grid'' of blanks by filling in each blank with
nonnegative numbers beginning at the top of each column and moving
down.  Each new entry can be any nonnegative number subject to two
restrictions.  First, the numbers must (non-strictly) ascend as one descends
a column and second the entry in each column must be less than or
equal to the number of entries to the right (not necessarily
above).  So for example,
$$\shortstack[c]{0\\2 \\6 \\8} \quad 
\shortstack[c]{4 \\5 \\5} \quad
\shortstack[c]{0\\0 \\0 \\0\\ 0}$$
is a tableaux.  But  
$$\shortstack[c]{0\\8 \\2 \\8} \quad 
\shortstack[c]{3 \\5 \\5} \quad
\shortstack[c]{0\\0 \\0 \\0\\ 0}\qquad \text{and} \qquad
\shortstack[c]{0\\2 \\8 \\8} \quad 
\shortstack[c]{3 \\5 \\6} \quad
\shortstack[c]{0\\0 \\0 \\0\\ 0}
$$
are not.

We shall consider all tableaux of a particular shape and assign to
that shape a polynomial in as many variables, $x_{i}$, as there are
columns.  We shall denote this polynomial by $[m_{1},m_{2},\cdots,
m_{n}]_{2}(x_{i})$ where the shape is $n$ columns of heights $m_{1}$
on the left followed by $m_{2}$ next to the right etc.  The polynomial
is found by summing over all the tableau of the given shape.  Each
summand is found by inserting ``$x_{i}+$'' in each entry of the $i$-th
column and multiplying all entries.  Using similar reasoning to that
in the proof of Proposition \ref{prop:itlogintermediate} we get
\begin{align}
\left(\frac{d}{dx}\right)^{k}\prod_{l=1}^{N}(\ell_{-l}(x))^{c_{l}}
=\underset{0 \leq j_{1},\cdots, j_{N}}
{\sum_{j_{1}+\cdots+j_{N}=k}}
[j_{1},\cdots, j_{N}]_{2}(c_{i})
\ell_{-N}(x)^{c_{N}}
\prod_{i=1}^{N-1}\ell_{-i}(x)^{c_{i}+\alpha_{i+1}} \label{eq:De}.
\end{align}
We shall specialize this calculation, but shall first revisit the
tableaux.  Notice that the rightmost column of a tableau is completely
determined by its length alone.  Therefore the piece of the tableau
polynomial due to the rightmost column may be factored out.  If look
at the remaining factor of $[m_{1},m_{2}, \cdots ,m_{n}]_{2}(x_{i})$
and set all the variables to zero we get an integer which we shall
call $(m_{1},m_{2}, \cdots ,m_{n-1};m_{n})_{2}$.  Notice that when
constructing a tableau of fixed shape each column is labelled
independently of the others.  Thus we have
\begin{prop}
For $m_{1}, \dots,m_{n} \geq 0$,
\begin{align*}
(m_{1},m_{2}, \cdots, m_{n-1};m_{n})_{2} =(m_{1};m_{2} +\cdots
+m_{n})_{2}\cdots(m_{n-2};m_{n-1}+m_{n})_{2}(m_{n-1};m_{n})_{2}.
\end{align*}
\label{prop:stirl2ind}
\end{prop}
\begin{flushright} $\square$ \end{flushright}
It is easy to see that
\begin{align*}
(m-n;n)_{2}=
\sum_{0 \leq i_{1} \leq i_{2} \leq \cdots  \leq i_{m-n} 
\leq n}i_{1}i_{2}\cdots i_{m-n}.
\end{align*}
By (\ref{stirl2}) we have,
\begin{align*}
(m-n;n)_{2}=\st{m}{n}.
\end{align*}
Now we can specialize \eqref{eq:De} to get for $N \geq 0$
\begin{align*}
\left(\frac{d}{dx}\right)^{k}\ell_{-N}(x)^{r}
&=\underset{0 \leq j_{1},j_{2},\cdots, j_{N}}
{\sum_{j_{1}+j_{2}+\cdots+j_{N}=m}}
r^{j_{N}}
(j_{1}, \cdots, j_{N-1};j_{N})_{2}
\ell_{-N}(x)^{r}
\prod_{i=1}^{N-1}\ell_{-i}(x)^{\alpha_{i+1}}\\
&=\underset{0 \leq j_{1},j_{2},\cdots, j_{N}}
{\sum_{j_{1}+\cdots+j_{N}=m}}
r^{j_{N}}
\left(\prod_{i=1}^{N-1}
(j_{i};\alpha_{i+1})_{2}
\right)
\ell_{-N}(x)^{r}
\prod_{i=1}^{N-1}\ell_{-i}(x)^{\alpha_{i+1}}\\
&=\underset{0 \leq j_{1},j_{2},\cdots, j_{N}}
{\sum_{j_{1}+\cdots+j_{N}=m}}
r^{j_{N}}
\left(\prod_{i=1}^{N-1}
\st{\alpha_{i}}{\alpha_{i+1}} 
\right)
\ell_{-N}(x)^{r}
\prod_{i=1}^{N-1}\ell_{-i}(x)^{\alpha_{i+1}}.
\end{align*}
Thus we get 
\begin{theorem} For $r \in \mathbb{C}$, $N \geq 0$,
\label{itexpfirstmeth}
\begin{align*}
\ell_{-N}(x+y)^{r}
%=e^{y\frac{d}{dx}}\ell_{-N}(x)^{r}
=\sum_{m \geq 0}\frac{y^{m}}{m!}
\underset{0 \leq j_{1},j_{2},\cdots, j_{N}}
{\sum_{j_{1}+\cdots+j_{N}=m}}
r^{j_{N}}
\left(\prod_{i=1}^{N-1}
\st{\alpha_{i}}{\alpha_{i+1}} 
\right)
\ell_{-N}(x)^{r}
\prod_{i=1}^{N-1}\ell_{-i}(x)^{\alpha_{i+1}}.
\end{align*}
\end{theorem}
\begin{flushright} $\qed$ \end{flushright}
\section{Formal analytic expansions: method 2}
In this section we begin to calculate formal analytic expansions of
$\ell_{N}(x+y)^{r}$ for $N \in \mathbb{Z}$, $r \in \mathbb{C}$
according to a second method.  Recall Remark \ref{rem:logexpsecret}
for the ``true meanings'' of these objects.  We shall begin with the
``small $N$'' cases treating in order $N=0,1,-1$ and $-2$ separately.
Then we shall deal with the general case.  The reader may skip ahead
to Section~\ref{sec:gencasemethod2} without any loss of generality.

\subsection{Case $N=0$, the case ``$x$''}
\label{sectioncaseN=0}
Here we first calculate $e^{y\frac{d}{dx}}\ell_{0}(x)$ and then using
the automorphism property, or alternatively the formal Taylor theorem,
we may find $e^{y\frac{d}{dx}}\ell_{0}(x)^{n}$ for $n \in \mathbb{N}$
by using a (nonnegative integral) binomial expansion.  We have
\begin{align}
\label{eq:aut1}
e^{y\frac{d}{dx}}\ell_{0}(x)^{n}=(\ell_{0}(x)+y)^{n}=\sum_{k \geq
0}\binom{n}{k}\ell_{0}(x)^{n-k}y^{k}.
\end{align}
We would like to extend this to the case where we have a general
complex exponent.  Of course, it is easy to do it directly, but we may
also use the result for nonnegative integral exponents to help us, a
method which is convenient in more difficult expansions.
\begin{prop}
\label{prop:method2case1}
For all $r \in \mathbb{C}$
\begin{align*}
e^{y\frac{d}{dx}}\ell_{0}(x)^{r}=(\ell_{0}(x)+y)^{r}=\sum_{m \geq
0}\binom{r}{m}\ell_{0}(x)^{r-m}y^{m}.
\end{align*}
\end{prop}
\begin{proof}
  Notice that $e^{y\frac{d}{dx}}\ell_{0}(x)^{r} \in
  \mathbb{C}\ell_{0}(x)^{r}[\ell_{0}(x)^{-1}][[y]]$ with coefficients
  being polynomials in $r$.  When $r$ is a nonnegative integer we have
  by (\ref{eq:aut1}) that the polynomial is
\begin{align*}
\binom{r}{m}=\frac{r(r-1)\cdots(r-m+1)}{m!},
\end{align*}
and since polynomials are determined by a finite number of values we
get the result.
\end{proof}

\subsection{Case $N=1$, the case ``$\log x $''}

We also have $\left(\frac{d}{dx}\right)^{l}
\ell_{1}(x)=(-1)^{l-1}(l-1)!\ell_{0}(x)^{-l}$ for
$l \geq 1$.  
%It is convenient as usual to 
%\begin{align*}
%\log (1+X)=\sum_{l \geq 1}\frac{(-1)^{l-1}}{l}X^{l}
%\end{align*}
%for any formal expression $X$ for which the expansion has finitely
%many terms contributing to the coefficient of each monomial.  We note
%that we now have two different types of formal logarithm, just as in
%\cite{HLZ}.
Recalling (\ref{eq:formallog}), we get that for $n \in \mathbb{N}$:
\begin{align}
\label{eq:method4 log}
e^{y\frac{d}{dx}}\ell_{1}(x)^{n}
&=\left(e^{y\frac{d}{dx}}\ell_{1}(x)\right)^{n} \nonumber \\
&=\left(\ell_{1}(x)
+\log \left(1+\frac{y}{\ell_{0}(x)}\right)\right)^{n} \nonumber\\
&=\sum_{j \geq 0}\binom{n}{j}\ell_{1}(x)^{n-j}
\left(\sum_{l \geq 1}\frac{(-1)^{l-1}}{l}
\left(\frac{y}{\ell_{0}(x)}\right)^{l}\right)^{j}\nonumber\\
&=\sum_{j \geq 0}\binom{n}{j}\ell_{1}(x)^{n-j}
\sum_{m \geq j}(-1)^{m-j}
\underset{m_{i} \geq 1}{\sum_{m_{1}+\cdots+m_{j}=m}}
\frac{1}{m_{1}\cdots m_{j}}
\left(\frac{y}{\ell_{0}(x)}\right)^{m}.
\end{align}
Then by arguing as in the proof of Proposition~\ref{prop:method2case1}
we immediately get the following.
\begin{prop}
For all $r \in \mathbb{C}$
\begin{align*}
e^{y\frac{d}{dx}}\ell_{1}(x)^{r}
=
\sum_{j \geq 0}\binom{r}{j}\ell_{1}(x)^{r-j}
\sum_{m \geq j}(-1)^{m-j}
\underset{m_{i} \geq 1}{\sum_{m_{1}+\cdots+m_{j}=m}}
\frac{1}{m_{1}\cdots m_{j}}
\left(\frac{y}{\ell_{0}(x)}\right)^{m}.
\end{align*}
\end{prop}
\begin{flushright} $\qed$ \end{flushright}

\subsection{Case $N=-1$, the case ``${\rm exp}\;x$''}
We have for $n \in \mathbb{N}$
\begin{align*}
e^{y\frac{d}{dx}}\ell_{-1}(x)^{n}&=
\left(e^{y\frac{d}{dx}}\ell_{-1}(x)\right)^{n}\\
&=\left(\ell_{-1}(x)e^{y}\right)^{n}\\
&=\ell_{-1}(x)^{n}(1+(e^{y}-1))^{n}\\
%&=\left(\ell_{-1}(x)+\ell_{-1}(x)(e^{y}-1)\right)^{n}\\
%&=\sum_{m \geq 0}\binom{n}{m}
%\ell_{-1}(x)^{n-m}\ell_{-1}(x)^{m}(e^{y}-1)^{m}\\
&=\ell_{-1}(x)^{n}\sum_{m \geq 0}\binom{n}{m}(e^{y}-1)^{m}.\\
\end{align*}
Then by arguing as in the proof of Proposition~\ref{prop:method2case1}
we immediately get the following.
\begin{prop}
For all $r \in \mathbb{C}$
\begin{align*}
e^{y\frac{d}{dx}}\ell_{-1}(x)^{r}
&=\ell_{-1}(x)^{r}\sum_{m \geq 0}\binom{r}{m}(e^{y}-1)^{m}.
\end{align*}
%where $\left(e^{y}\right)^{r}$ is
%interpreted as $\left(1+\left(e^{y}-1\right)\right)^{r}$ expanded
%using the binomial expansion convention.  
\end{prop}
\begin{flushright} $\qed$ \end{flushright}

\subsection{Case $N=-2$, the case ``${\rm exp}\;{\rm exp}\;x$''}
\label{sectioncaseN=-2}
In order to proceed as in the previous three examples we would like to
be able to easily calculate $e^{y\frac{d}{dx}}\ell_{-2}(x)$ and then
take the $m$-th power of the result.  However, it is just as difficult
to calculate $e^{y\frac{d}{dx}}\ell_{-2}(x)$ as
$e^{y\frac{d}{dx}}\ell_{-2}(x)^{m}$, since the answer already involves
two variables non-trivially.  We shall therefore use a different
strategy which we could have used in place of method 4 in the case
$N=-1$ and also in a sense cases $N=0,1$ although these cases are
roughly like initial cases.  We shall use a recursion formula which we
state here.
\begin{theorem} For $n \in \mathbb{Z}$ we have
\label{th:itlogrec}
\begin{align*}
\ell_{n+1}(x+y)&=\ell_{n+1}(x)+\log
\left(1+\left(\frac{\ell_{n}(x+y)-\ell_{n}(x)}
{\ell_{n}(x)}\right)\right)\\
 &\text{and} \\
\ell_{n}(x+y)&=\ell_{n}(x)e^{\left(\ell_{n+1}(x+y)-\ell_{n+1}(x)\right)}.
\end{align*}
\end{theorem}
\begin{flushright} $\qed$ \end{flushright} 

A proof of this is given in \cite{R2}.
\section{Formal analytic expansions: general case $N \in \mathbb{Z}$
  ``iterated logarithms and exponentials''}
\label{sec:gencasemethod2}
In this section we complete the calculation, using the second method,
of the formal analytic expansion of $\ell_{N}(x+y)^{r}$ for all $N \in
\mathbb{Z}$.  
\subsection{Iterated logarithms: second method}
Letting $n \in \mathbb{N}$ and $N \geq 0$, we can use
Theorem~\ref{th:itlogrec} to get:
\begin{align*}
\ell_{N}(x+y)^{n}&=\left(\ell_{N}(x)+\log 
\left(1+\left(\frac{\ell_{N-1}(x+y)-\ell_{N-1}(x)}
{\ell_{N-1}(x)}\right)\right)
\right)^{n}\\
=&\sum_{s \geq 0}\binom{n}{s}\ell_{N}(x)^{n-s} \cdot \\
&\,\,
\cdot
\left(
\log 
\left(
1+\ell_{N-1}(x)^{-1} \log 
\left(1+
\left(
\frac{\ell_{N-2}(x+y)-\ell_{N-2}(x)}
{\ell_{N-2}(x)}
\right)
\right)
\right)
\right)^{s}.
%&=\left(\ell_{N}(x)+
%\sum_{m \geq 1}\frac{(-1)^{m-1}}{m}
%\left(\frac{\ell_{N-1}(x+y)-\ell_{N-1}(x)}{\ell_{N-1}(x)}\right)^{m}
%\right)^{n}\\
%&=\sum_{s \geq 0}\binom{n}{s}\ell_{N}(x)^{n-s}
%\left(\sum_{m \geq 1}\frac{(-1)^{m-1}}{m}
%\left(\frac{\ell_{N-1}(x+y)-\ell_{N-1}(x)}{\ell_{N-1}(x)}\right)^{m}
%\right)^{s}\\
%&=\sum_{s \geq 0}\binom{n}{s}\ell_{N}(x)^{n-s}
%\sum_{k \geq 0}
%\frac{s!}{k!}
%{k \atopwithdelims[] s}
%\genfrac{[}{]}{1.5pt}{}{k}{s}
%(-1)^{k-s}
%\left(\frac{\ell_{N-1}(x+y)-\ell_{N-1}(x)}{\ell_{N-1}(x)}\right)^{k}\\
%&=\sum_{k,s \geq 0} \binom{n}{s}\ell_{N}(x)^{n-s}
%\frac{s!}{k!}
%{ k \atopwithdelims[] s}
%\genfrac{[}{]}{1.5pt}{}{k}{s}
%(-1)^{k-s}
%\ell_{N-1}(x)^{-k} \cdot \\
%& \quad \cdot \sum_{l \geq 0} 
%\binom{k}{l}\ell_{N-1}(x+y)^{l}(-\ell_{N-1}(x))^{k-l}\\
%&=\sum_{k,s,l,\geq 0}\binom{n}{s}
%\frac{s!}{k!}
%{k \atopwithdelims[] s}
%\genfrac{[}{]}{1.5pt}{}{k}{s}
%\binom{k}{l}
%(-1)^{s+l}\ell_{N}(x)^{n-s}\ell_{N-1}(x)^{-l}\ell_{N-1}(x+y)^{l}.
\end{align*}
\begin{remark} \rm
Compare with the calculation leading to (3.16) in \cite{HLZ}.  
\end{remark}
Iterating this and recalling (\ref{stirl1gen}), we get
\begin{align*}
\ell_{N}(x+y)^{n}&=
\sum_{0 \leq s_{N}}
\binom{n}{s_{N}}\ell_{N}(x)^{n-s_{N}}
\sum_{s_{N} \leq s_{N-1}}(-1)^{s_{N-1}-s_{N}}
\frac{s_{N}!}{s_{N-1}!}
\St{s_{N-1}}{s_{N}}
%{ s_{N-1} \atopwithdelims[] s_{N} }
\ell_{N-1}(x)^{-s_{N-1}} \cdot\\
&\cdot \sum_{s_{N-1} \leq s_{N-2}}(-1)^{s_{N-2}-s_{N-1}}
\frac{s_{N-1}!}{s_{N-2}!}
\St{s_{N-2}}{s_{N-1}}
%{ s_{N-2} \atopwithdelims[] s_{N-1} }
\ell_{N-2}(x)^{-s_{N-2}}
\cdots\\
&
\cdots
\sum_{s_{1} \leq s_{0}}(-1)^{s_{0}-s_{1}}
\frac{s_{1}!}{s_{0}!}
\St{s_{0}}{s_{1}}
%{ s_{0} \atopwithdelims[] s_{1} }
\ell_{0}(x)^{-s_{0}}
y^{s_{0}}\\
&=\ell_{N}(x)^{n}
\sum_{0 \leq s_{N} \leq s_{N-1} \leq \cdots \leq s_{1} \leq s_{0}}
\binom{n}{s_{N}}
(-1)^{s_{0}-s_{N}}
\frac{s_{N}!}{s_{0}!}
\prod_{i=1}^{N}
\St{s_{i-1}}{s_{i}}
%{ s_{i-1} \atopwithdelims[] s_{i}}
\prod_{i=0}^{N}
\ell_{i}(x)^{-s_{i}}
y_{0}^{s_{0}}.
\end{align*}

Then by arguing as in the proof of Proposition~\ref{prop:method2case1}
we immediately get the following.
\begin{theorem} For all $r \in \mathbb{C}$, $N \geq 0$
\label{itlogexpsecond}
\begin{align*}
\ell_{N}(x+y)^{r}
&=
\ell_{N}(x)^{r}
\sum_{0 \leq s_{N} <\cdots < s_{1} < s_{0} }
\,
\binom{r}{s_{N}}
\frac{s_{N}!}{s_{0}!}
\prod_{i=0}^{N-1}
\St{s_{i}}{s_{i+1}}
%{s_{i} \atopwithdelims[] s_{i+1}} 
%\genfrac{[}{]}{1.5pt}{}{s_{i}}{s_{i+1}}  
(-1)^{s_{0}-s_{N}}
\prod_{i=0}^{N}\ell_{i}(x)^{-s_{i}}
y^{s_{0}}.
\end{align*}
%\begin{align*}
%\ell_{N}(x+y)^{r}&=
%\sum_{s_{0}, \cdots, s_{N} \geq 0}
%\prod_{i=0}^{N-1}
%\genfrac{[}{]}{1.5pt}{}{s_{i}}{s_{i+1}}
%(-1)^{s_{0}+s_{N}}
%\frac{s_{N}!}{s_{0}!}
%\binom{r}{s_{N}}
%\ell_{N}(x)^{r}
%\prod_{i=0}^{N}\ell_{i}(x)^{-s_{i}}y^{s_{0}}.
%\end{align*}
\end{theorem}
\begin{flushright} $\square$ \end{flushright}

\subsection{Iterated exponentials: second method}
We use (\ref{th:itlogrec}) as in the last section, but this time the
second form so that we may iterate in the other direction.  We get for
$N \leq 0$ and $n \in \mathbb{N}$
\begin{align*}
\ell_{N}(x+y)^{n}
&=\left(\ell_{N}(x)e^{\left(\ell_{N+1}(x+y)-\ell_{N+1}(x)\right)}\right)^{n}\\
&=\ell_{N}(x)^{n}e^{n\left(\ell_{N+1}(x+y)-\ell_{N+1}(x)\right)}\\
&=\ell_{N}(x)^{n}
e^{-n\ell_{N+1}(x)}
e^{n\ell_{N+1}(x+y)}\\
&=\ell_{N}(x)^{n}
e^{-n\ell_{N+1}(x)}
e^{n\ell_{N+1}(x)e^{\left(\ell_{N+2}(x+y)-\ell_{N+2}(x)\right)}}\\
&=\ell_{N}(x)^{n}
e^{n\ell_{N+1}(x)\left(e^{\left(\ell_{N+2}(x+y)-\ell_{N+2}(x)\right)}-1\right)}\\
&=\ell_{N}(x)^{n}
e^{n\ell_{N+1}(x)
\left(e^{\ell_{N+2}(x)
\left(e^{\left(\ell_{N+3}(x+y)-\ell_{N+3}(x)\right)}-1\right)}-1\right)},
\end{align*}
at which point the iteration is clear (the factor with $n$ plays no
role after the first iteration).

It is convenient to write $N$ as a nonnegative number, so now,
recalling (\ref{stirl2gen}), we let $N \geq 0$ and $n \in \mathbb{N}$
to get
\begin{align*}
\ell_{-N}(x+y)^{n}
=&\ell_{-N}(x)^{n}
\sum_{0 \leq l_{1}}\frac{(n\ell_{-N+1}(x))^{l_{1}}}{l_{1}!} \cdot\\
& 
\cdot
\sum_{l_{1} \leq l_{2}}\frac{\ell_{-N+2}(x)^{l_{2}}l_{1}!}{l_{2}!}
\st{l_{2}}{l_{1}} 
\sum_{l_{2} \leq l_{3}}\frac{\ell_{-N+3}(x)^{l_{3}}l_{2}!}{l_{3}!}
\st{l_{3}}{l_{2}} 
\cdots\\
&\sum_{l_{N-2} \leq l_{N-1}}
\frac{\ell_{-1}(x)^{l_{N-1}}l_{N-2}!}{l_{N-1}!}
\st{l_{N-1}}{l_{N-2}} 
\sum_{l_{N-1} \leq l_{N}}
\frac{y^{l_{N}}l_{N-1}!}{l_{N}!}
\st{l_{N}}{l_{N-1}}\\
=&
\ell_{-N}(x)^{n}
\sum_{0 \leq l_{1} \leq l_{2} \leq \cdots \leq l_{N}}
n^{l_{1}}
\st{l_{2}}{l_{1}} 
\st{l_{3}}{l_{2}} 
\cdots 
\st{l_{N-1}}{l_{N-2}}
\st{l_{N}}{l_{N-1}} 
\cdot\\
& \cdot
\ell_{-N+1}(x)^{l_{1}}\ell_{-N+2}(x)^{l_{2}} \cdots \ell_{-1}(x)^{l_{N-1}}
\frac{y^{l_{N}}}{l_{N}!}.
\end{align*}
Then by arguing as in the proof of Proposition~\ref{prop:method2case1}
we immediately get the following.
\begin{theorem}
\label{itexpsecondmeth}
Let $N \geq 0$ and $r \in \mathbb{C}$ we get
\begin{align*}
\ell_{-N}(x+y)^{r}
=&
\ell_{-N}(x)^{r}
\sum_{0 \leq l_{1} \leq l_{2} \leq \cdots \leq l_{N}}
r^{l_{1}}
\prod_{i=1}^{N-1} 
\st{l_{i+1}}{l_{i}} 
\prod_{i=1}^{N-1}\ell_{-i}(x)^{l_{N-i}}
\frac{y^{l_{N}}}{l_{N}!}.
\end{align*}
\end{theorem}
\begin{flushright} $\square$ \end{flushright}

\section{Recurrences}
The coefficients of our various expansions satisfy linear recurrence
relations.  We shall indicate these recurrences for the low $N$ cases
giving the Stirling numbers of the first and second kinds and leave to
the interested reader any routine generalization.

\subsection{Logarithmic case}
Let $M(m,j)$ satisfy
\begin{align*}
\left(\frac{d}{dx}\right)^{m}\ell_{0}(x)^{s}\ell_{1}(x)^{r}
=
\sum_{j=0}^{m}
r(r-1) \cdots(r-j+1)
(-1)^{m-j}M(m,j)
\ell_{0}(x)^{s-m}\ell_{1}(x)^{r-j}.
\end{align*}
Two terms from
$\left(\frac{d}{dx}\right)^{m-1}\ell_{0}(x)^{s}\ell_{1}(x)^{r}$
contribute to each term of the next derivative.  Namely
\begin{align*}
r(r-1) \cdots (r-j+1)
(-1)^{m-1-j}M(m-1,j)
\ell_{0}(x)^{s-m+1}\ell_{1}(x)^{r-j}
\end{align*}
and
\begin{align*}
r(r-1) \cdots(r-j+2)
(-1)^{m-j}M(m-1,j-1)
\ell_{0}(x)^{s-m+1}\ell_{1}(x)^{r-j+1}
\end{align*}
contribute to 
\begin{align*}
r(r-1) \cdots(r-j+1)
(-1)^{m-j}M(m,j)
\ell_{0}(x)^{s-m}\ell_{1}(x)^{r-j},
\end{align*}
which yields 
\begin{align*}
r(r-1) \cdots(r-j+1)
&(-1)^{m-j}M(m,j)\\
=&
r(r-1) \cdots(r-j+1)
(-1)^{m-1-j}(s-m+1)M(m-1,j)\\
&+
r(r-1) \cdots(r-j+2)
(-1)^{m-j}(r-j+1)M(m-1,j-1),
\end{align*}
giving
\begin{align*}
M(m,j)
=
(m-1-s)M(m-1,j)+M(m-1,j-1),
\end{align*}
for $1 \leq j < m$
with boundary conditions easily seen to be given by
\begin{align*}
%M(0,j)=0 \qquad & j >0\\
M(m,0)=&(-s)(1-s) \cdots (m-1-s) \qquad  m > 0, \\
 \text{and} \quad
M(m,m)=&1\qquad  m \geq 0.
\end{align*} 
This immediately gives the following.
\begin{prop}
For all $r \in \mathbb{C}$
\begin{align*}
e^{y\frac{d}{dx}}\ell_{1}(x)^{r}
=\sum_{m \geq 0}
\sum_{n=0}^{m}
\binom{r}{n}
\frac{n!}{m!}
(-1)^{m-n}
\St{m}{n}
%{ m \atopwithdelims[] n }
\left(\frac{y}{\ell_{0}(x)}\right)^{m}\ell_{1}(x)^{r-n},
\end{align*}
where $\St{m}{n}$ 
%${ m \atopwithdelims[] n }$ 
is given by (\ref{rec1}) and (\ref{rec1b}).
%\begin{align*}
%{ m \atopwithdelims[] n }=(m-1){ m-1  \atopwithdelims[] n }+{ m-1 \atopwithdelims[] n-1 },
%\end{align*}
%for $1 \leq n < m$ with
%\begin{align*}
%{0 \atopwithdelims[] n}=0 \quad \text{for} \quad  n >0,
%\qquad 
%{m \atopwithdelims[] 0}=0 \quad \text{for}\quad m > 0,
%\qquad \text{and} \qquad
%{m \atopwithdelims[] m}=1  \quad \text{for}\quad m \geq 0.
%\end{align*}
\end{prop}
\subsection{Exponential case}
Let $N(m,j)$ satisfy
\begin{align*}
\left( \frac{d}{dx} \right) ^{m}
\ell_{-1}(x)^{s}\ell_{-2}(x)^{r}=
\sum_{j \geq 0}
r^{j}N(m,j)
\ell_{-1}(x)^{s+j}\ell_{-2}(x)^{r}.
\end{align*}
Two terms from $\left(\frac{d}{dx}\right)^{m-1}\ell_{-1}(x)^{s}\ell_{-2}(x)^{r}$
contribute to each term of the next derivative.  Namely
\begin{align*}
r^{j}N(m-1,j)
\ell_{-1}(x)^{s+j}\ell_{-2}(x)^{r}
\end{align*}
and
\begin{align*}
r^{j-1}N(m-1,j-1)
\ell_{-1}(x)^{s+j-1}\ell_{-2}(x)^{r}
\end{align*}
each contribute to
\begin{align*}
r^{j}N(m,j)
\ell_{-1}(x)^{s+j}\ell_{-2}(x)^{r},
\end{align*}
which yields
\begin{align*}
r^{j}N(m,j)=r^{j}(s+j)N(m-1,j)+r^{j-1}rN(m-1,j-1),
\end{align*}
giving
\begin{align*}
N(m,j)=(s+j)N(m-1,j)+N(m-1,j-1),
\end{align*}
for $1 \leq j < m$ with boundary conditions obviously given by
\begin{align*}
N(m,m)=1 \qquad &m \geq 0\\
 \text{and} \quad
N(m,0)=s^{m} \qquad &m > 0. 
\end{align*}
This immediately gives the following.
\begin{prop}
For all $r \in \mathbb{C}$
\begin{align*}
e^{y\frac{d}{dx}}\ell_{-2}(x)^{r}
=\sum_{m \geq 0}
\sum_{n=0}^{m}
\frac{r^{n}}{m!}
\st{m}{n}
\ell_{-1}(x)^{n}\ell_{-2}(x)^{r}y^{m},
\end{align*}
where $\st{m}{n}$ is given by (\ref{rec2}) and (\ref{rec2b}).
%where
%\begin{align*}
%{ m \st{}{} n}=n{ m-1 \st{}{} n} +{m-1 \st{}{} n-1}
%\end{align*}
%for $1 \leq n < m$ with boundary conditions given by
%\begin{align*}
%{0 \st{}{} n }=0 \quad \text{for} \quad n >0,
%\qquad
%{m \st{}{} 0}=0 \quad \text{for} \quad m > 0,
%\qquad \text{and} \quad
%{m \st{}{} m}=1 \quad \text{for} \quad m \geq 0.  
%\end{align*}
\end{prop}

\section{Identities}
Because we used two different methods to calculate the formal analytic
expansions we may equate the results to get combinatorial identities.
We have been anticipating some of these results already and so have
already used the same notation for two different expressions for the
Stirling numbers.  Therefore, temporarily, in this section, Stirling
numbers of the first kind will be denoted by $\St{m}{n}_{1}$ 
%${m \atopwithdelims[] n}_{1}$ 
when given by (\ref{stirl1}) and by $\St{m}{n}_{2}$
%${m \atopwithdelims[] n}_{2}$
when given by (\ref{eq:stirling1}).  Similarly, Stirling numbers of
the second kind will be denoted by ${\st{m}{n}}_{1}$ when given by
(\ref{stirl2}) and by ${\st{m}{n}}_{2}$ when given by
(\ref{eq:stirling2}).

\subsection{Logarithmic case}
We wish to equate the expansions of Theorems \ref{itlogexpfrst} and
\ref{itlogexpsecond}.  We have for $r \in \mathbb{C}$ and $N \geq 0$,
\begin{align*}
\ell_{N}(x+y)^{r}&=
\sum_{m \geq 0}\frac{y^{m}}{m!}
\underset{0 \leq j_{0},j_{1},\cdots, j_{N}}
{\sum_{j_{0}+j_{1}+\cdots+j_{N}=m}}
j_{N}!\binom{r}{j_{N}}
(-1)^{\alpha_{0}-\alpha_{N}}
\left(\prod_{i=0}^{N-1}
\St{\alpha_{i}}{\alpha_{i+1}}_{1}
%{ \alpha_{i} \atopwithdelims[] \alpha_{i+1}}_{1}
\right)
\ell_{N}(x)^{r}\prod_{i=0}^{N}\ell_{i}(x)^{-\alpha_{i}}\\
&=
\ell_{N}(x)^{r}
\sum_{0 \leq s_{N} <\cdots < s_{1} < s_{0} }
\quad
\binom{r}{s_{N}}
\frac{s_{N}!}{s_{0}!}
\prod_{i=0}^{N-1}
\St{s_{i}}{s_{i+1}}_{2}
%{s_{i} \atopwithdelims[] s_{i+1}}_{2}
%\genfrac{[}{]}{1.5pt}{}{s_{i}}{s_{i+1}}
(-1)^{s_{0}+s_{N}}
\prod_{i=0}^{N}\ell_{i}(x)^{-s_{i}}
y^{s_{0}}.
\end{align*}
Recalling that 
\begin{align*}
\alpha_{i}=j_{i}+ \cdots+j_{N},
\end{align*}
we get 
\begin{align*}
j_{i}=\alpha_{i}-\alpha_{i+1},
\end{align*}
for $0 \leq i \leq N-1$ and $\alpha_{N}=j_{N}$.  Then we may rewrite
the first expression for $\ell_{N}(x+y)^{r}$ as
\begin{align*}
\ell_{N}(x+y)^{r}&=
\sum_{0 \leq \alpha_{N} \leq \cdots \alpha_{2} \leq \alpha_{1} \leq \alpha_{0}}
\alpha_{N}!
\binom{r}{\alpha_{N}}
(-1)^{\alpha_{0}-\alpha_{N}}
\left(\prod_{i=0}^{N-1}
\St{\alpha_{i}}{\alpha_{i+1}}_{1}
%{ \alpha_{i} \atopwithdelims[] \alpha_{i+1}}_{1}
\right)
\ell_{N}(x)^{r}\prod_{i=0}^{N}\ell_{i}(x)^{-\alpha_{i}}
\frac{y^{\alpha_{0}}}{\alpha_{0}!},
\end{align*}
and now equating coefficients, we get
\begin{align*}
\prod_{i=0}^{N-1}
\St{s_{i}}{s_{i+1}}_{1}
%{ s_{i} \atopwithdelims[] s_{i+1}}_{1}
=
\prod_{i=0}^{N-1}
\St{s_{i}}{s_{i+1}}_{2}
%{s_{i} \atopwithdelims[] s_{i+1}}_{2}
,
\end{align*}
which of course, is only interesting in the $N=1$ case, which gives the
classical identity:
\begin{align*}
\frac{m!}{n!}
\underset{i_{l} \geq 1}{\sum_{i_{1}+\cdots+ i_{s}=m}}
\frac{1}{i_{1}\cdots i_{n}}
=\sum_{0 \leq t_{1} < t_{2} < \cdots < t_{m-n}< m}t_{1} \cdots t_{m-n},
\end{align*}
for $1 \leq n \leq m$.  Well actually, it gives more than this, because it
gives an automatic proof of Proposition~\ref{prop:inditstirl1}.  That
is, equating coefficients also gives
\begin{align*}
(\alpha_{0}-\alpha_{1}, \dots , \alpha_{N-1}-\alpha_{N};\alpha_{N})_{1}
=
\prod_{i=0}^{N-1}
\St{\alpha_{i}}{\alpha_{i+1}}_{2}
%{\alpha_{i} \atopwithdelims[] \alpha_{i+1}}_{2}
.
\end{align*}
Then the $N=1$ case gives
\begin{align*}
(\alpha_{0}-\alpha_{1};\alpha_{1})_{1}
=
\St{\alpha_{0}}{\alpha_{1}}_{2}
%{\alpha_{0} \atopwithdelims[] \alpha_{1}}_{2}
,
\end{align*}
which yields
\begin{align*}
(\alpha_{0}-\alpha_{1}, \dots , \alpha_{N-1}-\alpha_{N};\alpha_{N})_{1}
=
\prod_{i=0}^{N-1}
(\alpha_{i}-\alpha_{i+1};\alpha_{i+1})_{1},
\end{align*}
or 
\begin{align*}
(j_{0}, \dots , j_{N-1};j_{N})_{1}
=
\prod_{i=0}^{N-1}
(j_{i};j_{i+1}+ \cdots + j_{N})_{1},
\end{align*}
which is Proposition~\ref{prop:inditstirl1}.
\subsection{Exponential case}
We wish to equate the expansions of Theorems \ref{itexpfirstmeth} and
\ref{itexpsecondmeth}.  We have for $r \in \mathbb{C}$ and $N \geq 0$,
\begin{align*}
\ell_{-N}(x+y)^{r}=&\sum_{m \geq 0}\frac{y^{m}}{m!}
\underset{0 \leq j_{1},j_{2},\cdots, j_{N}}
{\sum_{j_{1}+\cdots+j_{N}=m}}
r^{j_{N}}
\left(\prod_{i=1}^{N-1}
{\st{\alpha_{i}}{\alpha_{i+1}}}_{1}
%(j_{i};\alpha_{i+1})_{2}
\right)
\ell_{-N}(x)^{r}
\prod_{i=1}^{N-1}\ell_{-i}(x)^{\alpha_{i+1}}\\
=&
\ell_{-N}(x)^{r}
\sum_{0 \leq l_{1} \leq l_{2} \leq \cdots \leq l_{N}}
r^{l_{1}}
\prod_{i=1}^{N-1} 
{\st{l_{i+1}}{l_{i}}}_{2}
\prod_{i=1}^{N-1}\ell_{-i}(x)^{l_{N-i}}
\frac{y^{l_{N}}}{l_{N}!}.
\end{align*}
As in the last section, we may substitute $\alpha_{i}$'s for the
$j_{i}$'s to get
\begin{align*}
&\sum_{0 \leq \alpha_{N} \leq \cdots \alpha_{2} \leq \alpha_{1}}
\frac{y^{\alpha_{1}}}{\alpha_{1}!}
r^{\alpha_{N}}
\left(\prod_{i=1}^{N-1}
{\st{\alpha_{i}}{\alpha_{i+1}}}_{1}
%(\alpha_{i}-\alpha_{i+1};\alpha_{i+1})_{2}
\right)
\ell_{-N}(x)^{r}
\prod_{i=1}^{N-1}\ell_{-i}(x)^{\alpha_{i+1}}\\
=&
\ell_{-N}(x)^{r}
\sum_{0 \leq l_{1} \leq l_{2} \leq \cdots \leq l_{N}}
r^{l_{1}}
\prod_{i=1}^{N-1} 
{\st{l_{i+1}}{l_{i}}}_{2}
\prod_{i=1}^{N-1}\ell_{-i}(x)^{l_{N-i}}
\frac{y^{l_{N}}}{l_{N}!},
\end{align*}
which gives
\begin{align*}
&\sum_{0 \leq l_{1} \leq \cdots l_{N-1} \leq l_{N}}
\frac{y^{l_{N}}}{l_{N}!}
r^{l_{1}}
\left(\prod_{i=1}^{N-1}
{\st{l_{i+1}}{l_{i}}}_{1}
%(l_{i+1}-l_{i};l_{i})_{2}
\right)
\ell_{-N}(x)^{r}
\prod_{i=1}^{N-1}\ell_{-i}(x)^{l_{N-i}}\\
=&
\ell_{-N}(x)^{r}
\sum_{0 \leq l_{1} \leq l_{2} \leq \cdots \leq l_{N}}
r^{l_{1}}
\prod_{i=1}^{N-1} 
{\st{l_{i+1}}{l_{i}}}_{2}
\prod_{i=1}^{N-1}\ell_{-i}(x)^{l_{N-i}}
\frac{y^{l_{N}}}{l_{N}!},
\end{align*}
yielding
\begin{align*}
\prod_{i=1}^{N-1}
\st{l_{i+1}}{l_{i}}_{1}
%(l_{i+1}-l_{i};l_{i})_{2}
=
\prod_{i=1}^{N-1} \st{l_{i+1}}{l_{i}}_{2},
\end{align*}
which, of course, is only interesting in the $N=2$ case, which gives
the classical identity:
\begin{align*}
\frac{m!}{n!}\sum_{i_{1}+i_{2}+\cdots i_{n}=m}\frac{1}{i_{1}!i_{2}!\cdots i_{n}!}
=
\sum_{0 \leq i_{1} \leq i_{2} \leq \cdots  \leq i_{m} 
\leq n}i_{1}i_{2}\cdots i_{m},
\end{align*}
for $1 \leq n \leq m$.  Well, actually, as in the logarithm case, it gives
more, because it gives an automatic proof of
Proposition~\ref{prop:stirl2ind}.  That is, equating coefficients also
gives
\begin{align*}
(\alpha_{1}-\alpha_{2},\dots, \alpha_{N-1}-\alpha_{N};\alpha_{N})_{2}
=
\prod_{i=1}^{N-1} \st{\alpha_{i}}{\alpha_{i+1}}_{2}.
\end{align*}
Then the $N=2$ case gives 
\begin{align*}
(\alpha_{1}-\alpha_{2};\alpha_{2})_{2}
=
\st{\alpha_{1}}{\alpha_{2}}_{2},
\end{align*} 
which yields 
\begin{align*}
(\alpha_{1}-\alpha_{2},\dots, \alpha_{N-1}-\alpha_{N};\alpha_{N})_{2}
=
\prod_{i=1}^{N-1}
(\alpha_{i}-\alpha_{i+1};\alpha_{i+1})_{2},
\end{align*}
or
\begin{align*}
(j_{1},\dots, j_{N-1};j_{N})_{2}
=
\prod_{i=1}^{N-1}
(j_{i};j_{i+1}+\cdots +j_{N})_{2},
\end{align*}
which is Proposition~\ref{prop:stirl2ind}.

\noindent {\small \sc Department of Mathematics, Rutgers University,
Piscataway, NJ 08854} 
\\ {\em E--mail
address}: thomasro@math.rutgers.edu
\end{document}